\documentclass[]{interact}

\usepackage{epstopdf}
\usepackage[caption=false]{subfig}
\usepackage{float} 

\usepackage[numbers,sort&compress]{natbib}
\bibpunct[, ]{[}{]}{,}{n}{,}{,}
\makeatletter
\def\NAT@def@citea{\def\@citea{\NAT@separator}}
\makeatother
\usepackage{xcolor}

\theoremstyle{plain}
\newtheorem{theorem}{Theorem}[section]

\newtheorem{proposition}[theorem]{Proposition}
\theoremstyle{definition}

\theoremstyle{remark}
\newtheorem{remark}{Remark}

\newcommand{\R}{\mathbb{R}}
\newcommand{\Om}{\Omega}

\begin{document}


\title{Spatially Limited Immune Access Creates Tumour Refugia in Hepatocellular Carcinoma}

\author{
\name{Jiguang Yu\textsuperscript{a}\thanks{This author contributed equally to this work. Email: jyu678@bu.edu} and Louis Shuo Wang\textsuperscript{b}\thanks{This author contributed equally to this work. CONTACT Louis Shuo Wang. Email: wang.s41@northeastern.edu}}
\affil{\textsuperscript{a}College of Engineering, Boston University, Boston, 02215, MA, USA; \textsuperscript{b}Department of Mathematics, Northeastern University, Boston, 02115, MA, USA}
}

\maketitle

\begin{abstract}
Spatially restricted immune access can undermine tumour control even when total immune recruitment appears sufficient. We develop a nondimensional reaction--diffusion--chemotaxis model for hepatocellular carcinoma that couples tumour growth and immune-mediated killing with effector diffusion, saturating chemokine-dependent recruitment, and migration along an effective CXCL9/CXCL10/CXCL11--CXCR3 signal. We establish nonnegativity, uniform tumour and mass bounds, and global boundedness of classical solutions in one dimension and, in arbitrary dimensions, when chemokine production is tumour-driven. Analysis of the homogeneous dynamics shows that the threshold $\sigma_0>\delta$ is only local and does not preclude bistable tumour persistence. A mode-wise dispersion relation identifies stationary finite-wavelength and oscillatory instabilities, whose critical sensitivities, dominant modes, growth rates, and frequencies are quantitatively recovered by conservative finite-volume simulations. Using matched initial states and identical mean recruitment, we show that a well-mixed model may predict clearance while margin-limited recruitment preserves a stable interior tumour refuge. Chemotaxis improves interior effector access and reduces, but need not eliminate, this refuge. Spatial-overlap biomarkers further connect model-generated phenotypes to pathology and spatial-omics observables. These results distinguish immune abundance from effective spatial access and identify access-limited recruitment as a mechanism of incomplete tumour control.
\end{abstract}


\begin{amscode}
35K57; 92C17; 35B36; 65M08; 92C50
\end{amscode}

\section{Introduction}\label{sec:introduction}

Hepatocellular carcinoma (HCC) is governed not only by the proliferative
capacity of malignant hepatocytes but also by the spatial organization of the tumour microenvironment (TME) in which they reside
\cite{bray2024global,llovet2022immunotherapies,bruni2020immune}. Cancer
progression more generally emerges from coupled processes operating across
molecular, cellular, and tissue scales
\cite{beeghly2023measuring,desoyer2025computational}. Cytokines, chemokines, and metabolic signals regulate cell division, survival, migration, and immune activation \cite{tiwari2025molecular}, while tumour, immune, stromal, and endothelial populations interact through direct contact and diffusible mediators \cite{kaminska2015role,joshi2026dynamic,he2022extracellular}. These interactions generate tissue-scale heterogeneity in tumour expansion, immune infiltration and exclusion, necrosis, invasion fronts, and treatment response \cite{giuliani2025immune,liang2026separation,biswas2022inference}. Tumour progression is therefore not simply a temporal growth process: local spatial interactions can determine macroscopic disease states and therapeutic outcomes.

This multiscale view has made mathematical and computational oncology an important framework for connecting mechanistic hypotheses with observable tumour behaviour. Quantitative cancer models have evolved from low-dimensional population descriptions toward integrative, spatial, and multiscale approaches for studying tumour initiation, progression, heterogeneity, and treatment response
\cite{anderson2008integrative,altrock2015mathematics,yu2026microscopic,mcdonald2023computational,gopukumar2026multiscale}. In parallel, community roadmaps have emphasized reproducible quantitative cancer modelling \cite{rockne20192019,wang2025multi}, and emerging mechanistic-learning
frameworks seek to combine interpretable dynamics with clinical and data-driven
inference \cite{metzcar2024review,lan2025shallow,liu2025bidirectional,de2025radiation,laslo2025mechanistic}. Within this programme, spatial models are especially valuable because they connect cell--signal interactions with tissue architecture rather than representing the TME only through spatial averages.

Classical tumour--immune models provide the kinetic foundation for this work. Since the ordinary-differential-equation model of Kuznetsov~et~al. \cite{kuznetsov1994nonlinear,gao2022rolling}, nonlinear systems have been used to investigate tumour dormancy, immune escape, recurrence, immune recruitment, and immunotherapy
\cite{kirschner1998modeling,de2001mathematical,mehdizadeh2023targeting,yu2026beyond,barrera2025impact}; see also the synthesis of Eftimie~et~al.~\cite{eftimie2011interactions,yu2026physics}. Such well-mixed descriptions clarify the temporal competition between tumour growth and immune-mediated elimination and yield useful kinetic thresholds for tumour control. Their limitation is structural: after tissue architecture is averaged out, an immune-poor tumour is indistinguishable from a tumour containing many immune cells that remain spatially separated from malignant cells. This distinction is increasingly important because the biological consequences of immune abundance depend not only on cell number, but also on localization and accessibility.

This distinction is particularly important in HCC, where immune evasion is spatially organized. Cytotoxic cells may be abundant at the tissue scale or enriched near a tumour boundary while failing to penetrate tumour-dense regions. Such mislocalization can produce immune-cold or immune-excluded compartments even when bulk measurements indicate substantial immune abundance \cite{galon2006type,yu2026age,mariathasan2018tgfbeta,liu2023identification}.
The density, localization, and organization of tumour-infiltrating lymphocytes are associated with tumour phenotype and clinical outcome \cite{lopez2025biological,yu2026pattern,brummel2023tumour}, motivating the classification of tumours into immune-hot, altered, and cold states \cite{galon2019approaches}. In immune exclusion, effector T cells are present within the TME but remain separated from tumour cells by stromal, extracellular matrix, or other microenvironmental barriers \cite{joyce2015t,yu2024extracellular,wang2026algebraic,bruni2023cancer}. Accordingly, tumour control depends not only on whether immune killing can outpace malignant proliferation, but also on whether cytotoxic cells can reach, persist within, and overlap tumour-rich regions.

Spatial mathematical models provide a natural means of resolving this access problem. Reaction--diffusion equations describe proliferation, death, signal production and degradation, and passive spatial spreading \cite{kondo2010reaction}; taxis terms describe directed cell movement in response to chemical gradients and thereby convert local signalling into tissue-scale redistribution. This is especially relevant for immune cells responding to chemokine fields in the TME \cite{mempel2024chemokines}. Reaction--diffusion--chemotaxis systems can thus connect molecular signalling and cell motility to tumour invasion, immune infiltration, immune exclusion, coexistence, and pattern formation
\cite{matzavinos2004mathematical,tao2024global,wang2026elliptic,li2026global,ai2015reaction,ke2022analysis,kiselev2022chemotaxis}. Related spatial models have addressed tumour-induced angiogenesis \cite{yu2026rigorous,hadjigeorgiou2024hybrid,anderson1998continuous,agosti2023image},
avascular tumour growth \cite{wang2026breakdown,roose2007mathematical,ferreira2002reaction,chaplain1996avascular},
and cytotoxic T-lymphocyte responses to solid tumours \cite{matzavinos2004mathematical}. Together, these studies show that diffusible signals and directed migration can generate heterogeneous infiltration, travelling structures, and irregular tumour distributions that homogeneous models cannot represent. Continuum formulations can also be complemented by
agent-based and hybrid approaches, including PhysiCell models of multicellular
tumour dynamics \cite{ghaffarizadeh2018physicell,cai2026optimal,savic2022heterogeneous} and
coupled PDE--agent-based models resolving interactions among hypoxia, angiogenesis, vascular heterogeneity, and treatment resistance \cite{wang2025analysis1}.

Chemokine-mediated recruitment and migration provide a direct mechanism for spatially heterogeneous immune access. Chemokines govern the recruitment, positioning, and retention of immune effector cells within tissues \cite{nagarsheth2017chemokines,chaudhary2025role,liang2025global,foeng2022harnessing}.
In HCC and other solid tumours, the CXCL9/CXCL10/CXCL11--CXCR3 axis is associated with type-I immune activation and effector-cell trafficking \cite{reschke2022cxcl9}. We represent this axis through the lumped tissue-scale field
\[
w(x,t)=\text{effective CXCR3-ligand chemokine activity},
\]
which aggregates CXCL9, CXCL10, and CXCL11 rather than resolving their individual kinetics. This abstraction isolates two effects of the chemokine field that are biologically related but mathematically distinct: it increases local effector recruitment through a saturating response and directs existing effectors up chemokine gradients through chemotaxis. Recruitment changes the local supply of immune cells; chemotaxis redistributes that supply. Separating these mechanisms is central to determining whether sufficient total immune input produces effective access to tumour-rich tissue.

The increasing availability of spatially resolved experimental measurements makes this distinction empirically relevant. Multiplex imaging, spatial transcriptomics, and modern computational pathology provide maps of tumour regions, immune abundance, cell--cell proximity, exclusion patterns, and signalling activity
\cite{rao2021exploring,yu2026full,moffitt2022emerging,palla2022spatial}. Such measurements reveal spatial phenotypes but do not, by themselves, determine the dynamical mechanisms that generate them. A mechanistic model can complement such data by representing tumour growth, immune recruitment, chemokine production, directed migration, and restricted access as explicit processes. Conversely, overlap- and threshold-based summaries calculated from model fields can be compared with segmented histology or spatial-omics maps. This provides a route from observed immune-hot, immune-cold, or immune-excluded states to experimentally testable mechanistic hypotheses.

The resulting model belongs to the Keller--Segel class of chemotaxis systems \cite{keller1970initiation,hillen2009user,wang2025analysis,horstmann20031970}.
Because chemotactic aggregation can generate steep gradients and, in some regimes, loss of boundedness, the balance between taxis, diffusion, reaction, and logistic damping is a central analytical issue \cite{winkler2010aggregation,tao2012boundedness,wang2026damage,arumugam2021keller}.
Recent studies demonstrate how signal coupling and nonlinear response determine global existence, boundedness, and pattern formation in related systems \cite{tao2024global,wang2026elliptic,li2026global,ai2015reaction,ke2022analysis,kiselev2022chemotaxis}. Such control is essential here because a mathematically uncontrolled chemotactic amplification would undermine the interpretation of computed densities as physically meaningful cell populations. Chemotactic concentration also creates a numerical challenge, motivating conservative finite-volume methods, positivity-preserving central-upwind schemes, and conservative upwind finite-element formulations \cite{filbet2006finite,chertock2008second,yu2026size,saito2007conservative}. These considerations motivate the conservative finite-volume treatment adopted here, together with upwind chemotactic fluxes and discrete positivity and residual diagnostics.

The mechanistic foundation of the present study is the nondimensional tumour--immune--chemokine reaction--diffusion--chemotaxis model developed in our previous work \cite{liu2026computational}. That study established the core local interaction structure and its initial equilibrium, stability, and numerical framework. Here we retain that structure but address a different question: whether spatially heterogeneous immune supply and transport can reverse a treatment conclusion obtained from a well-mixed description. Our central hypothesis is that \emph{spatially limited immune access can preserve tumour refugia that are invisible to a well-mixed recruitment model, whereas chemokine-guided migration can reduce, but need not eliminate, those refugia}. The analysis therefore treats immune recruitment, directed migration, and tumour--immune overlap as distinct determinants of tumour control.

The present study makes four principal contributions. First, we strengthen the analytical foundation of the model by establishing preservation of nonnegativity, a uniform tumour bound, and global existence with uniform boundedness of classical solutions in one space dimension and, in arbitrary dimension, when chemokine production is tumour-driven ($\gamma=0$). We also show that the well-mixed condition $\sigma_0>\delta$ is a local tumour-free stability threshold rather than a global clearance criterion: bistability can preserve a tumour-positive attractor for sufficiently large initial burden. Second, we resolve stationary and oscillatory routes to spatial instability through a mode-wise cubic dispersion relation. Eigenvector-seeded simulations quantitatively validate the unstable wavelengths, growth rates, onset thresholds, and the complex-pair frequency of the oscillatory route, linking linear predictions to stationary hot--cold domains and time-dependent immune-access waves.

Third, we introduce spatially restricted immune recruitment and construct a matched PDE--ODE comparison in which the averaged initial state, total immune supply, and kinetic parameters are held fixed. This controlled comparison reveals a recruitment--access mismatch: the well-mixed system clears the tumour, whereas the spatial system retains a stable interior refuge. Parameter mapping shows that the mismatch persists over an extended intervention region, and sensitivity experiments demonstrate that chemotaxis partially restores interior access without necessarily eliminating the refuge. Boundary-driven chemokine delivery provides a complementary negative control and does not robustly reproduce the refuge mechanism. Fourth, we define immune-hot, immune-cold, tumour--immune exclusion, and chemokine--access mismatch
biomarkers that translate model fields into quantities comparable, in principle, with pathology and spatial-omics maps. Together, these results connect mathematical well-posedness, homogeneous and spatial stability, verified computation, treatment response, and experimentally observable TME heterogeneity within a single mechanistic framework.

The remainder of the paper is organized as follows. Section~\ref{sec:methods} formulates the core and extended tumour--immune--chemokine models, introduces the spatial recruitment and chemokine-delivery protocols, defines the matched well-mixed comparator and spatial biomarkers, and summarizes the analytical and numerical procedures. Section~\ref{sec:results} presents the well-posedness and equilibrium results, examines bistability and the limitations of the well-mixed clearance threshold, validates the stationary and oscillatory spatial-instability mechanisms, and quantifies the emergence and mitigation of access-limited tumour refugia. Section~\ref{sec:discussion} interprets the mathematical and biological implications, relates the model outputs to pathology and spatial-omics measurements, and discusses limitations and future extensions. The supporting proofs, equilibrium derivations, numerical implementation and convergence diagnostics, and parameter definitions are provided in Appendices~\ref{app:analysis}--\ref{app:parameters}.

\section{Materials and Methods}\label{sec:methods}

We model a bounded tissue region $\Om\subset\R^n$, $n=1,2$, over $t\ge0$. The
primary variables are the tumor-cell density $u(x,t)$, the cytotoxic
effector-cell density $v(x,t)$, and the effective chemokine activity $w(x,t)$;
an optional suppressive TME signal $s(x,t)$ appears in the extended model
(Table~\ref{tab:variables_tme}). Their biological meanings and the
corresponding spatial observables are summarized below.

\begin{table}[htbp]
\centering
\caption{Model variables and biological interpretation.}
\label{tab:variables_tme}
\renewcommand{\arraystretch}{1.2}
\begin{tabular}{lll}
\toprule
\textbf{Variable} & \textbf{Meaning} & \textbf{Biological interpretation}\\
\midrule
$u(x,t)$ & Tumor-cell density & HCC malignant-cell burden\\
$v(x,t)$ & Effector immune-cell density & CD8$^+$ T / NK-like cytotoxic activity\\
$w(x,t)$ & Chemokine signal & Effective CXCL9/10/11--CXCR3 recruitment axis\\
$s(x,t)$ & Suppressive TME signal & TGF-$\beta$, PD-L1/CD73, Tregs, MDSCs\\
\bottomrule
\end{tabular}
\end{table}

The baseline mechanistic model used in this study is the
nondimensional tumour--immune--chemokine system introduced in our
previous work~\cite{liu2026computational}. We reproduce the system
here for completeness, because the present study uses it as the
baseline model for a different set of questions concerning spatial
accessibility, intervention protocols, and spatially resolved
biomarkers:
\begin{equation}
\label{eq:core_pde}
\begin{aligned}
u_t &= d_1\Delta u+u(1-u-v),\\
v_t &= d_2\Delta v-\xi\nabla\cdot(v\nabla w)
+\sigma_0+\sigma_1\frac{w}{1+w}-\delta v-\beta uv,\\
w_t &= d_3\Delta w+\alpha u+\gamma uv-\ell w,
\end{aligned}
\end{equation}
with $d_1,d_2,d_3>0$, $\xi\ge0$, $\sigma_0,\sigma_1,\beta,\alpha,\gamma\ge0$,
and $\delta,\ell>0$. Here $u(1-u-v)$ represents logistic tumor growth
with immune-mediated killing $-uv$; the chemotactic flux
$-\xi\nabla\cdot(v\nabla w)$ models directed effector migration up chemokine
gradients; $\sigma_0$ is baseline recruitment and
$\sigma_1 w/(1+w)$ is saturating chemokine-dependent recruitment;
$\alpha u+\gamma uv$ describes tumor- and interaction-induced chemokine
production; and $-\ell w$ accounts for chemokine clearance. The model is
nondimensional. To incorporate additional TME-mediated suppression without
altering the structure of the core system, we also consider the extended model
\begin{equation}
\label{eq:extended_pde}
\begin{aligned}
u_t &= d_1\Delta u+u(1-u-v),\\
v_t &= d_2\Delta v-\nabla\cdot\!\big(\chi(w)v\nabla w\big)
+\sigma_0+\sigma_1\tfrac{w}{1+w}-\delta v-\beta uv-\eta sv,\\
w_t &= d_3\Delta w+\alpha u+\gamma uv-\ell w,\\
s_t &= d_4\Delta s+a_su+b_suv-\ell_s s,
\end{aligned}
\end{equation}
with $\chi(w)=\xi$ or the saturating form
$\chi(w)=\xi/(1+w)^m$, $m\ge0$. The three-variable model
\eqref{eq:core_pde} is used for the main results.

Unless stated otherwise we impose homogeneous no-flux boundary conditions,
\begin{equation}
\label{eq:noflux_core}
\nabla u\cdot n=\nabla v\cdot n=\nabla w\cdot n=0,\qquad x\in\partial\Om,
\end{equation}
with nonnegative initial data $u_0,v_0,w_0\ge0$. Since $\nabla w\cdot n=0$, the
total immune-cell flux $J_v=-d_2\nabla v+\xi v\nabla w$ also satisfies
$J_v\cdot n=0$; thus the imposed conditions imply no net tumor-cell,
immune-cell, or chemokine flux through the tissue boundary. These conditions
define the baseline closed-tissue setting, after which we introduce
intervention protocols that modify immune or chemokine supply while leaving the
underlying transport and reaction mechanisms unchanged.

For uniform immune-recruitment escalation, we use a time-dependent baseline,
$\sigma_0(t)=\sigma_{0,\min}+(\sigma_{0,\max}-\sigma_{0,\min}),t/T$ for
$0\le t\le T$, then held fixed. To represent recruitment entering through
vascularized tissue margins rather than uniformly, we instead use the spatially
fixed profile
\begin{equation}
\label{eq:sigma0_margin}
\sigma_0(x)=
\begin{cases}
\sigma_0^{\mathrm{hi}}, & \operatorname{dist}(x,\partial\Om)\le r_m,\\
0, & \operatorname{dist}(x,\partial\Om)>r_m,
\end{cases}
\end{equation}
with margin width $r_m>0$ and margin level $\sigma_0^{\mathrm{hi}}>0$. The
chemokine field and chemotaxis are unchanged; only the supply of effectors is
spatially localized. Its spatial mean
$\overline{\sigma_0}=|\Om|^{-1}\int_\Om\sigma_0(x)\,dx$ is the value supplied to
the well-mixed comparator. As a complementary intervention, a spatially
nonuniform chemokine delivery is modelled by an inhomogeneous Dirichlet
condition $w|_{\partial\Om}=B_w(t)$ or a flux condition
$-d_3\nabla w\cdot n=q_w(t)$. These protocols allow us to separate changes in
the total amount of immune or chemokine supply from the spatial accessibility
through which that supply acts.

To isolate the effect of spatial structure, we compare \eqref{eq:core_pde} with
the corresponding well-mixed system
\begin{equation}
\label{eq:ode_comparator}
\dot u = u(1-u-v),\quad
\dot v = \sigma_0+\sigma_1\tfrac{w}{1+w}-\delta v-\beta uv,\quad
\dot w = \alpha u+\gamma uv-\ell w.
\end{equation}
The comparison is constructed so that differences in outcome cannot be
attributed to different initial conditions or total parameter levels. Thus,
the ODE comparator and the PDE are always initialized from the \emph{same}
state: the ODE is started from the spatial average of the PDE initial data,
$(\overline{u_0},\overline{v_0},\overline{w_0})
=|\Om|^{-1}\int_\Om(u_0,v_0,w_0)\,dx$, and integrated with the spatial average
of any spatially varying parameter (e.g.\ $\overline{\sigma_0}$). Any difference
in predicted outcome is therefore attributable to spatial structure and
transport, rather than to a difference in initial tumor burden, total immune
supply, or kinetic parameters. This construction is used throughout when
assessing whether spatially limited access produces outcomes that are not
predicted by the well-mixed description.

The analytical and numerical procedures are designed to connect these model
comparisons with the underlying stability and pattern-forming mechanisms. We
prove positivity, a priori bounds, and global existence with uniform
boundedness (Appendix~\ref{app:analysis}), characterize the homogeneous
equilibria (Appendix~\ref{app:equilibria}), and derive the linear stability of
the tumor-free equilibrium and the coexistence dispersion relation. For each
Neumann eigenvalue $\mu_k$ the linearized mode matrix yields a cubic
$P_k(\lambda)=\lambda^3+a_1\lambda^2+a_2\lambda+a_3$; Routh--Hurwitz gives the
stationary route $a_3(\mu_k)<0$ and the oscillatory route
$a_1(\mu_k)a_2(\mu_k)<a_3(\mu_k)$ with $a_1,a_2,a_3>0$
(Appendix~\ref{app:analysis}). Simulations use a conservative finite-volume
scheme with first-order upwind chemotaxis flux in 1D, and an IMEX scheme in 2D
(implicit diffusion via the discrete cosine transform; explicit upwind
chemotaxis and reaction); see Appendix~\ref{app:numerics}. Positivity and
grid-refinement diagnostics are recorded for every run. Together, the analytical and
computational procedures provide both qualitative control of the model and
quantitative verification of its spatial instability and treatment responses.

The resulting spatial fields are intended not only as state variables for the
mechanistic analysis but also as synthetic representations of experimentally
observable TME maps. The model variables and derived biomarkers therefore map
onto measurable pathology and spatial-omics layers (Table~\ref{tab:model_data_mapping}).

\begin{table}[htbp]
\centering
\caption{Model-generated spatial layers and possible pathology or spatial-omics
analogues.}
\label{tab:model_data_mapping}
\renewcommand{\arraystretch}{1.25}
\begin{tabular}{p{0.16\textwidth}p{0.37\textwidth}p{0.37\textwidth}}
\toprule
\textbf{Model layer} & \textbf{Spatial data analogue} & \textbf{Mechanistic interpretation}\\
\midrule
$u(x,t)$ &
Tumor-cell or tumor-region probability map from H\&E, multiplex imaging, or
tumor-marker segmentation &
Spatial distribution of malignant HCC burden.\\
$v(x,t)$ &
Cytotoxic immune-cell density from CD8, granzyme-B, NK-cell, or related markers &
Spatial distribution of effector immune access.\\
$w(x,t)$ &
Spatial CXCL9/CXCL10/CXCL11--CXCR3 pathway activity from spatial transcriptomics
or pathway-inference maps &
Effector-attracting chemokine activity and directional recruitment signal.\\
$s(x,t)$ &
Suppressive pathway activity (TGF-$\beta$, CD274/PD-L1, NT5E/CD73, Treg, myeloid,
or macrophage-associated signatures) &
Spatially organized immunosuppressive TME activity.\\
$H_{\rm hot},H_{\rm cold}$ &
Hot/cold tumor-region fractions from segmented tumor and immune maps &
Threshold-dependent summaries of local tumor--immune overlap.\\
$I_{\rm excl}$ &
Immune-exclusion score from tumor and immune spatial maps &
Threshold-free measure of tumor--immune spatial separation.\\
$M_{\rm chem}$ &
Mismatch between chemokine/pathway activity and immune infiltration &
Indicator of chemokine-positive but immune-poor tumor regions.\\
\bottomrule
\end{tabular}
\end{table}

With tumor and immune thresholds $\theta_u,\theta_v>0$, we define the immune-hot
and immune-cold fractions of tumor-positive tissue by
\begin{equation}
H_{\rm hot}=\frac{|{u>\theta_u,\,v>\theta_v}|}{|{u>\theta_u}|},\qquad
H_{\rm cold}=\frac{|{u>\theta_u,\,v\le\theta_v}|}{|{u>\theta_u}|},
\end{equation}
and introduce the threshold-free indices
\begin{equation}
I_{\rm excl}=1-\frac{\int_\Om uv\,dx}{\|u\|_{L^2}\|v\|_{L^2}},\qquad
M_{\rm chem}=\frac{\int_\Om uw\,dx}{\|u\|_{L^2}\|w\|_{L^2}}
-\frac{\int_\Om uv,dx}{\|u\|_{L^2}\|v\|_{L^2}}.
\end{equation}
Large $I_{\rm excl}$ indicates spatial separation of tumor and effector cells;
positive $M_{\rm chem}$ indicates tumor regions co-localize more with chemokine
than with effectors (a chemokine--access mismatch).

The threshold-dependent fractions $H_{\rm hot}$ and $H_{\rm cold}$ serve as
interpretable phenotype summaries, analogous to classifying segmented tissue
regions as immune-infiltrated or immune-poor; because their values depend on the
chosen thresholds, the thresholds are reported explicitly for every simulation.
In contrast, $I_{\rm excl}$ and $M_{\rm chem}$ are threshold-free overlap
measures, and we use them as the primary quantitative biomarkers for comparing
model-generated spatial states with pathology or spatial-omics maps. In this
way, the same mathematical framework links the biological variables and
intervention protocols to analytical stability properties, numerical
experiments, and spatially resolved observables.

\section{Results}\label{sec:results}

For nonnegative initial data the core model preserves nonnegativity,
$u,v,w\ge0$, and the tumor density obeys
$0\le u(x,t)\le\max\{1,\|u_0\|_{L^\infty}\}$
(Appendix~\ref{app:analysis}, Propositions~\ref{app:prop:positivity}--%
\ref{app:prop:u_bound}). Moreover, the classical solution exists globally and is
uniformly bounded: unconditionally in one space dimension
(Theorem~\ref{app:thm:1d}), and in every dimension when chemokine production is
tumor-driven ($\gamma=0$, Proposition~\ref{app:prop:gamma0}); in two dimensions
with $\gamma>0$, global existence holds together with the a priori bounds under a smaller chemotactic sensitivity. The
extended model preserves nonnegativity of $s$ as well. These guarantees provide
the analytical foundation for interpreting the solutions as biologically
meaningful TME dynamics and are consistent with the numerical positivity
diagnostics reported below.

Having established that the model remains in the biologically admissible
state space, we first examine the corresponding well-mixed dynamics to
determine what can and cannot be inferred from population-level immune
recruitment. The tumor-free equilibrium is $E_0=(0,\sigma_0/\delta,0)$, and
linearization shows it is locally asymptotically stable iff
$\sigma_0>\delta$: baseline recruitment must exceed immune turnover. This is a
local condition. Integrating the well-mixed system in a bistable regime
($\beta=2,\sigma_1=0.2,\alpha=0.5,\gamma=0.5,\delta=0.5,\ell=1$) with
$\sigma_0=0.58>\delta=0.5$, the scalar coexistence equation $F(u_*)=0$ has two
interior roots, $u_*\approx0.068$ (a saddle) and $u_*\approx0.640$ (a stable
coexistence state), so $E_0$ coexists with a tumor-positive attractor
(Fig.~\ref{fig:bistability}B). Trajectories started near $E_0$ ($u_0=0.02$)
clear ($u(t_{\rm end})=\mathcal O(10^{-14})$), whereas moderate or established
tumors ($u_0=0.3,0.6$) converge to $u_*\approx0.640$
(Fig.~\ref{fig:bistability}A). Thus $\sigma_0>\delta$ guarantees only local
clearance of small perturbations; the outcome from an established tumor
depends on the initial burden. Accordingly, when we report ODE clearance below
we verify it numerically from the same averaged initial state used for the PDE,
rather than inferring it from the threshold.

\begin{figure}[htbp]
\centering
\includegraphics[width=0.95\textwidth]{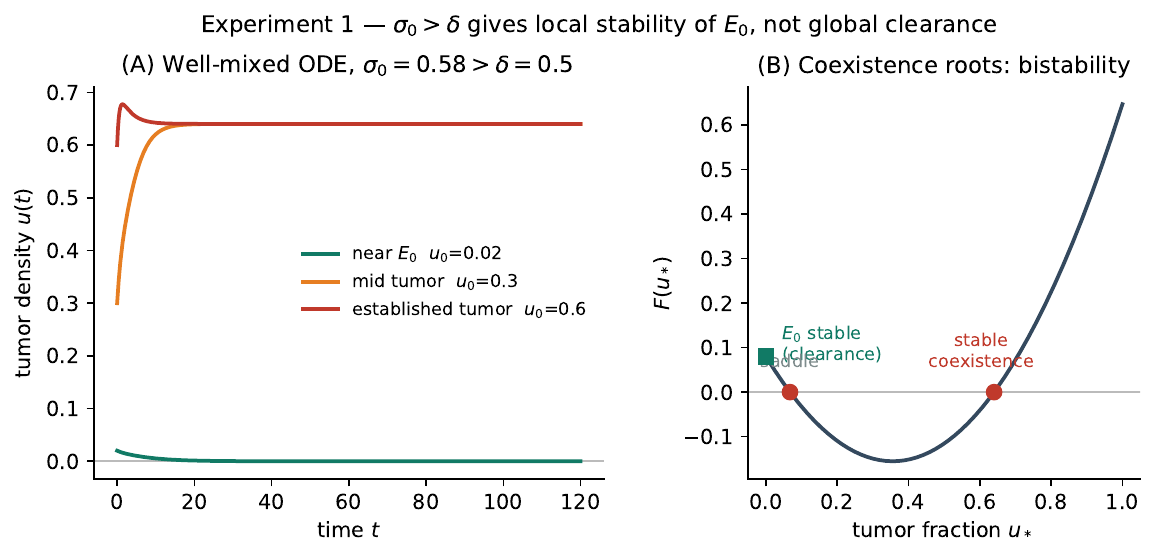}
\caption{Well-mixed bistability despite local stability of $E_0$. (A) ODE trajectories for $\sigma_0=0.58>\delta=0.5$: the near-$E_0$ start clears; moderate and established tumors converge to the coexistence state $u_*\approx0.640$. (B) $F(u_*)$ showing the stable $E_0$, a saddle root $u_*\approx0.068$, and a stable coexistence root $u_*\approx0.640$.\label{fig:bistability}}
\end{figure}

The limitations of the local well-mixed threshold become more pronounced once
spatial degrees of freedom are retained. At the coexistence equilibrium
$E_*\approx(0.760,0.240,0.259)$ of the regime
$\sigma_0=0.2,\sigma_1=0.5,\alpha=0.1,\gamma=1,\beta=1,\delta=0.5,\ell=1,
d_1=0.01,d_2=0.1,d_3=1$, $\xi=20$ on $\Om=(0,10)$, the predicted growth rate
$\lambda_k^{\rm pred}=\max_j\operatorname{Re}\lambda_j(M_k)$ is negative for
$k=0$ and positive for a finite band $k=4,\dots,10$, peaking at $k=7$ with
$\lambda_7^{\rm pred}\approx0.484$; high modes restabilize, consistent with the
cubic dominance $a_3(\mu_k)\sim d_1d_2d_3\mu_k^3$, so the wavelength is intrinsic
rather than mesh-determined. Two checks confirm the relation: the spectrum of
the discrete finite-volume Jacobian at $E_*$ matches the analytic per-mode
eigenvalues to four figures (global maximal rate $0.4839$ discrete vs.\ $0.4839$
analytic); and eigenvector-seeded single-mode simulations recover
$\lambda_k^{\rm num}$ to within $3\times10^{-4}$ of
$\lambda_k^{\rm pred}$ across the band (Fig.~\ref{fig:dispersion}A,B). From
random perturbations the nonlinear PDE saturates into persistent hot--cold
domains at the predicted wavenumber (Fig.~\ref{fig:dispersion}C). The
stationary route is thus quantitatively predictive, showing that the spatial
model does not merely introduce additional degrees of freedom but produces a
specific, analytically identifiable instability.

\begin{figure}[htbp]
\centering
\includegraphics[width=0.98\textwidth]{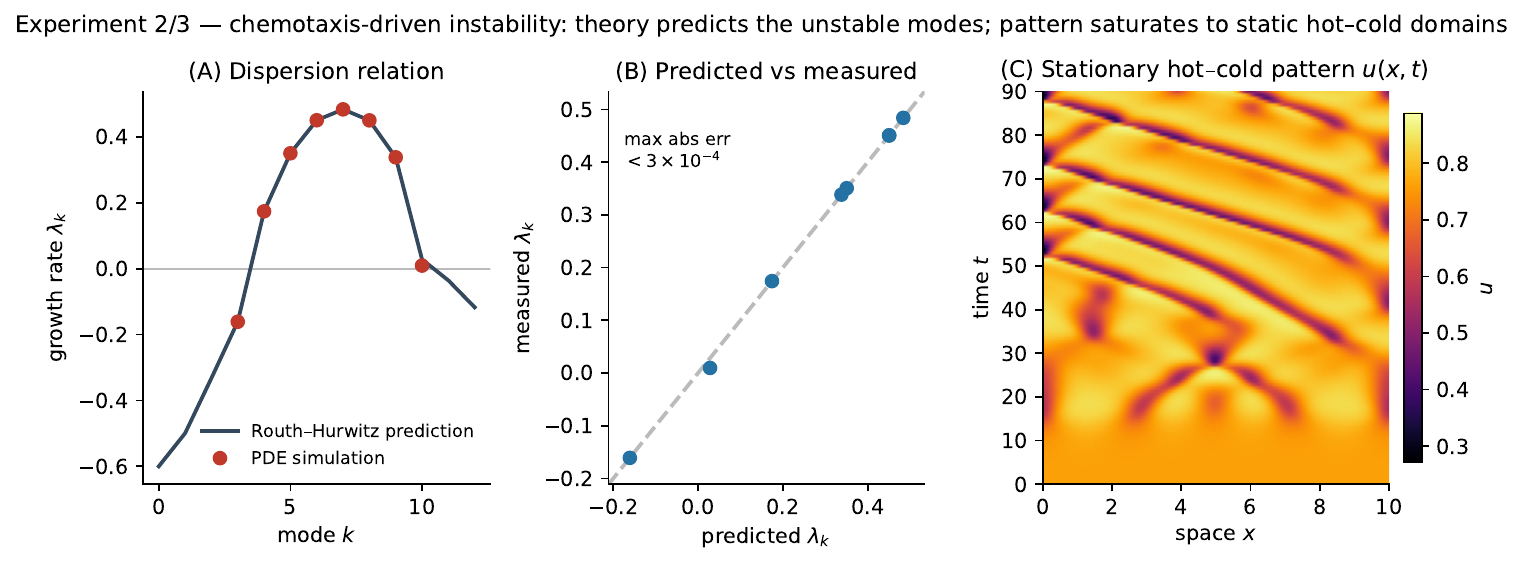}
\caption{Stationary dispersion-relation verification. (A) Analytic dispersion relation with simulation-measured growth rates (markers). (B) Predicted vs.\ measured rates (max absolute error $<3\times10^{-4}$). (C) Nonlinear saturation into persistent hot--cold domains at the fastest-growing wavenumber.\label{fig:dispersion}}
\end{figure}

Beyond individual growth rates, the same dispersion relation predicts the
\emph{onset} of patterning. Writing
$\Lambda(\xi)=\max_{k\ge1}\lambda_k^{\mathrm{pred}}$ for the maximal growth rate
over inhomogeneous modes, the coexistence state is spatially stable while
$\Lambda(\xi)<0$ and unstable once $\Lambda(\xi)>0$; in the regime above this
gives a critical sensitivity $\xi_c\approx16.63$. Simulations started from
small random perturbations confirm this threshold sharply: the final pattern
amplitude remains at the noise level
($\mathrm{std}(u)\lesssim10^{-9}$) for $\xi\le15$ and rises abruptly to
$\mathcal O(10^{-2})$ once $\xi$ exceeds $\xi_c$
(Fig.~\ref{fig:threshold}). The analytically predicted instability boundary
therefore coincides with the numerically observed onset of hot--cold patterning,
providing a direct quantitative connection between the linear theory and the
nonlinear simulations.

\begin{figure}[htbp]
\centering
\includegraphics[width=0.62\textwidth]{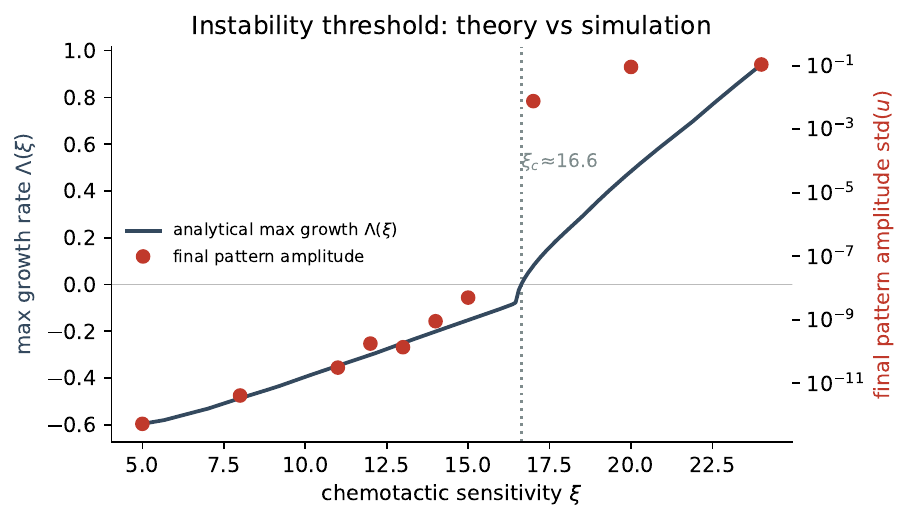}
\caption{Instability threshold, theory versus simulation. The analytic maximal
    growth rate $\Lambda(\xi)$ (curve) crosses zero at $\xi_c\approx16.6$ (dotted line); the final simulated pattern amplitude (markers, right log axis) switches from noise level to $\mathcal O(10^{-2})$ exactly at $\xi_c$.\label{fig:threshold}}
\end{figure}

The spatial theory predicts not only stationary pattern formation but also a
distinct oscillatory route. This second mechanism was realized at
$E_*\approx(0.511,0.489,0.636)$
($\sigma_0=0.3,\sigma_1=0.5,\alpha=1,\gamma=0.5,\beta=1,\delta=0.5,\ell=1,
d_1=0.01,d_2=0.1,d_3=1,\xi=20$). The band $k=3,\dots,9$ is unstable through the
determinant condition with $a_1,a_2,a_3>0$ and genuinely complex eigenvalues;
the fastest mode is $k=6$ with
$\lambda_6\approx0.189+1.682\,i$ (period $\approx3.74$). Eigenvector-seeded
simulation reproduces both parts of the complex pair: growth rate $0.179$
(predicted $0.189$) and angular frequency $1.680$ (predicted $1.682$), with a
sinusoid at the predicted frequency fitting the normalized oscillation at
$R^2=0.996$ (Fig.~\ref{fig:oscillatory}A). The saturated nonlinear state
sustains time-dependent immune-access waves (Fig.~\ref{fig:oscillatory}B).
Thus the two theoretically predicted instability routes are numerically
distinguishable, producing qualitatively different spatial states: static
hot--cold domains through the stationary route and dynamic immune-access waves
through the oscillatory route.

\begin{figure}[htbp]
\centering
\includegraphics[width=0.98\textwidth]{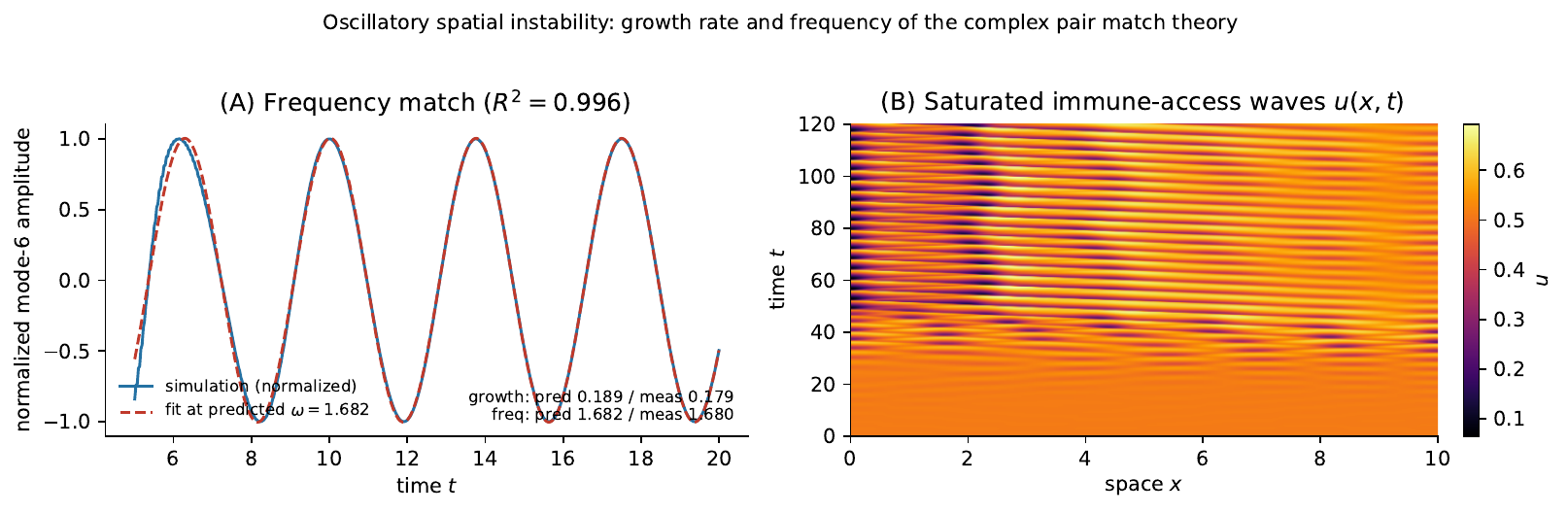}
\caption{Oscillatory spatial instability (second Routh--Hurwitz route). (A) Normalized mode-6 amplitude from the PDE with a sinusoid fitted at the predicted frequency $\omega=1.682$ ($R^2=0.996$); growth rate and frequency both match the complex-pair prediction. (B) Saturated state: persistent, time-dependent immune-access waves $u(x,t)$.\label{fig:oscillatory}}
\end{figure}

These instability results establish that spatial organization can fundamentally
alter TME dynamics, but they do not by themselves show that spatial structure
can change a treatment decision. We therefore next impose the
shared-initial-condition protocol and ask whether spatial organization of
immune access can reverse a well-mixed treatment prediction. The robust
mechanism is not metastable survival of a pattern under a uniform recruitment
increase --- under uniform recruitment, once
$\overline{\sigma_0}>\delta$ the spatial model cleared together with the ODE in
every regime examined --- but spatial limitation of immune \emph{access}.
Using the margin profile \eqref{eq:sigma0_margin} with
$\sigma_0^{\rm hi}=1.6$, $r_m=2$ on $\Om=(0,10)$ (and
$\sigma_1=0.5,\alpha=0.5,\gamma=0.5,\beta=1,
\delta=0.5,\ell=1,d_1=0.01,d_2=0.08,d_3=1,\xi=8$), the spatial mean is
$\overline{\sigma_0}=0.64>\delta=0.5$, so the matched comparator predicts
clearance --- confirmed numerically, $u_{\rm ODE}(t_{\rm end})\approx10^{-15}$.

Under identical total recruitment, the spatial model instead settles into a
stable interior tumor refuge: effectors accumulate at the recruiting margins
and are depleted in the interior, where the tumor persists with interior mean
density $\overline u_{\rm int}\approx0.38$ and interior mass $\approx2.30$,
while the margin tumor mass falls to $\approx10^{-3}$
(Fig.~\ref{fig:divergence}). The outcome is grid-stable (interior mass changes
by $2.8\times10^{-3}$, below $1\%$, from $N=200$ to $N=400$) and preserves
nonnegativity ($m_{\min}\approx5\times10^{-11}$). Because the ODE and PDE share
the same averaged initial state and total recruitment, the surviving interior
tumor is attributable solely to the spatial organization of immune access:
\emph{spatially averaged recruitment can predict clearance, while spatially limited access preserves an interior tumor refuge}.

\begin{figure}[htbp]
\centering
\includegraphics[width=0.98\textwidth]{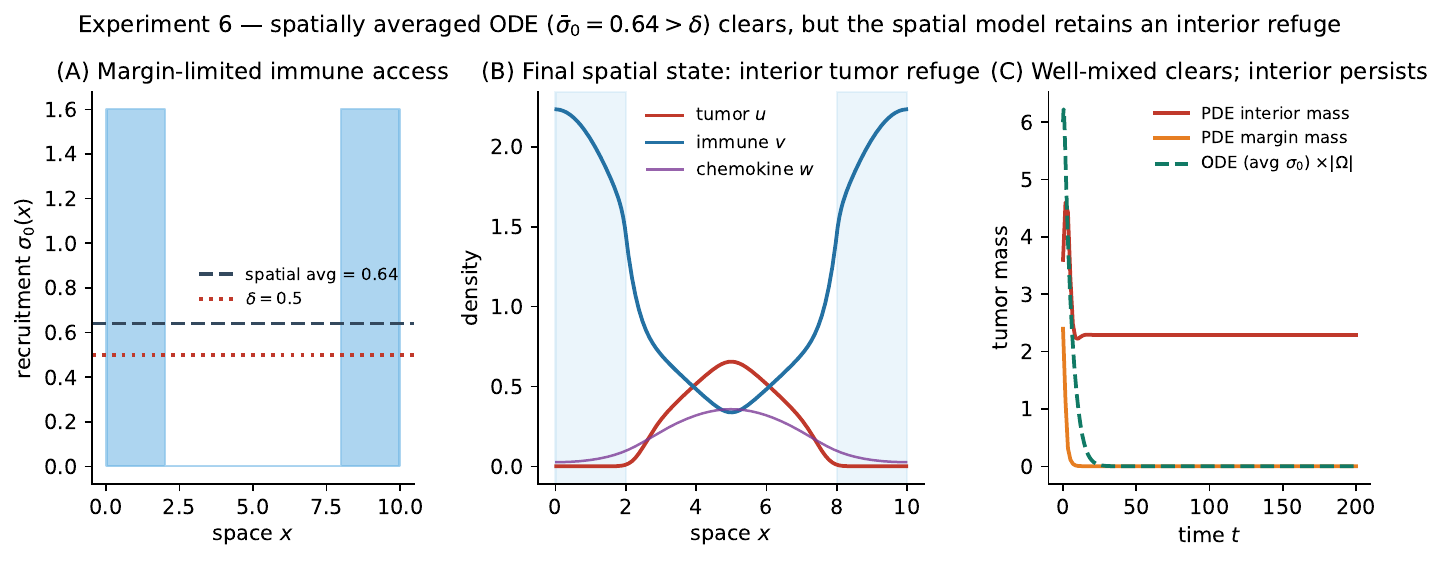}
\caption{Spatially limited immune recruitment creates ODE--PDE treatment divergence. (A) Margin-localized recruitment $\sigma_0(x)$ with spatial mean $\overline{\sigma_0}=0.64>\delta=0.5$. (B) Final PDE state: interior tumor refuge with margin-confined effector accumulation. (C) The spatially averaged ODE clears from the same averaged initial state, while the PDE interior tumor mass persists.\label{fig:divergence}}
\end{figure}

The same effect persists across a broader parameter region rather than arising
from a single parameter choice. Varying the margin recruitment level
$\sigma_0^{\rm hi}$ and the margin half-width $r_m$ over a grid, we computed
both the well-mixed outcome from the averaged recruitment
$\overline{\sigma_0}$ and the interior refuge mass of the spatial model
(Fig.~\ref{fig:refuge_phase}). The averaged ODE clears whenever
$\overline{\sigma_0}>\delta$ (dashed contour), yet across the entire clearing
region the spatial model retains a nonzero interior refuge, whose mass
decreases monotonically as the margin widens --- that is, as immune access
reaches deeper into the tissue. The refuge vanishes only when the recruiting
margin is wide enough to cover the interior. The single-point result of
Fig.~\ref{fig:divergence} is therefore representative of an extended region of
parameter space in which average recruitment sufficiency and effective spatial
access diverge.

\begin{figure}[htbp]
\centering
\includegraphics[width=0.66\textwidth]{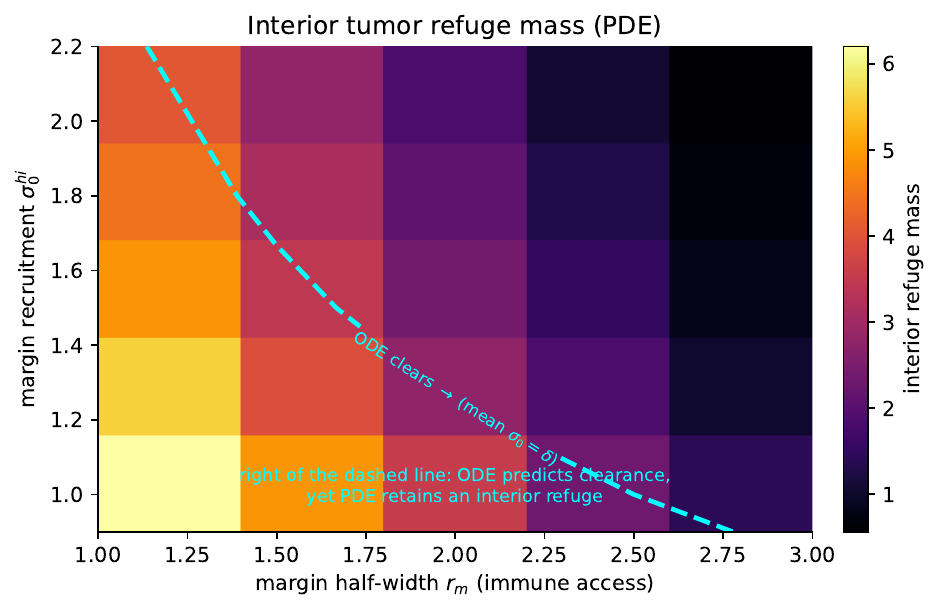}
\caption{Refuge phase diagram. Colour shows the spatial-model interior tumor refuge mass over margin recruitment $\sigma_0^{\rm hi}$ and margin half-width $r_m$; the dashed contour is the well-mixed clearance boundary $\overline{\sigma_0}=\delta$. To its right the averaged ODE predicts clearance, yet the spatial model retains an interior refuge that shrinks with increasing access (larger $r_m$).\label{fig:refuge_phase}}
\end{figure}

The preceding results identify limited recruitment as the principal mechanism
behind the ODE--PDE divergence. We next varied chemotactic sensitivity with the
margin-recruitment profile fixed to determine whether directed migration can
restore the lost access. Increasing $\xi$ from $0$ to $16$ reduced the interior
tumor mass monotonically from $4.59$ to $1.28$ and raised interior mean effector
density from $0.25$ to $0.80$ (Fig.~\ref{fig:chemotaxis_access}). Because the
interior tumor produces chemokine, directed migration transports
marginally recruited effectors up the resulting gradient and into the interior.
Chemokine-guided migration is thus a partial \emph{remedy} for access-limited
delivery rather than a cause of the refuge; it improves interior access but
cannot, alone, overcome a recruitment supply confined to the margins (a
residual refuge persists even at the largest $\xi$). This is consistent with
the dispersion results, since in both settings chemotaxis reorganizes effectors
in space rather than changing their number.

\begin{figure}[htbp]
\centering
\includegraphics[width=0.85\textwidth]{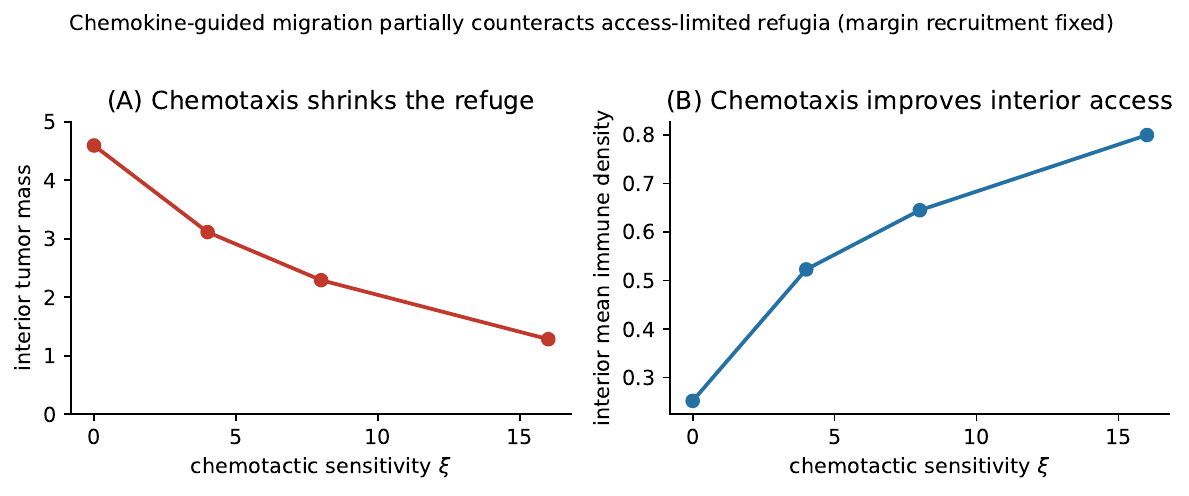}
\caption{Role of chemotaxis with margin recruitment fixed. (A) Interior tumor mass decreases monotonically with chemotactic sensitivity $\xi$. (B) Interior mean effector density increases correspondingly. Chemotaxis partially counteracts, but does not eliminate, the access-limited refuge.\label{fig:chemotaxis_access}}
\end{figure}

A complementary intervention is to manipulate the chemokine field directly at
the tissue boundary. We therefore also tested boundary-driven chemokine
delivery, imposing $w|_{\partial\Om}=B_w$ (ghost cells;
Appendix~\ref{app:numerics}) on a tumor-positive tissue with sub-threshold
uniform recruitment. In contrast to margin-limited recruitment, it did not
robustly produce a clean margin-clearance/interior-refuge separation. Two
effects limit it. First, the recruitment response is saturating,
$\sigma_1 w/(1+w)\le\sigma_1$, and acts only within the chemokine penetration
length $\sqrt{d_3/\ell}$: when this exceeds the margin, the interior is also
rescued and the tissue clears globally; when it is short, the margin is
incompletely covered. Second, with active chemotaxis the boundary gradient
draws effectors into a thin edge layer (local density exceeding $10^2$) that
starves the adjacent sub-margin. Across the delivery strengths, penetration
depths, and chemotactic sensitivities examined, boundary chemokine delivery
produced global clearance, global persistence, or non-monotone edge-concentrated
profiles, but not a clean interior refuge. The robust driver of interior
refugia is therefore spatial limitation of effector \emph{recruitment} rather
than of chemokine alone. This distinction further supports the interpretation
that the key therapeutic bottleneck is not simply the presence of a chemokine
signal, but the spatial availability of cytotoxic effectors.

\section{Discussion}\label{sec:discussion}

This study indicates that immune recruitment and immune access are distinct
determinants of tumor control in a spatially organized HCC microenvironment.
Increasing recruitment can be sufficient in a well-mixed model yet insufficient
in tissue when recruitment is spatially restricted to accessible regions such as
vascularized margins. In the central simulation the spatially averaged
recruitment satisfies the well-mixed control condition and the matched ODE
comparator clears from the same averaged initial state, yet the PDE preserves an
interior tumor refuge because effectors accumulate at the margins and do not
adequately reach the interior. Immune abundance and immune access are therefore
not equivalent: a tissue may receive enough total immune input to clear tumor cells in the corresponding well-mixed model while spatial delivery constraints leave interior compartments immune-cold.

This distinction also clarifies the role of chemokine-guided migration. Because
the interior tumor produces chemokine, directed migration transports marginally
recruited effectors into the tumor-rich interior, so that stronger chemotaxis
\emph{reduces} the refuge and raises interior effector density
(Fig.~\ref{fig:chemotaxis_access}); the residual refuge does not vanish even at
the largest $\xi$ tested. Chemotaxis thus acts as a partial remedy for
access-limited delivery rather than as a cause of the refuge. At the same time,
the same chemotactic mechanism can, in other parameter regimes, destabilize
homogeneous coexistence into stationary hot--cold domains or oscillatory
immune-access waves (Figs.~\ref{fig:dispersion}--\ref{fig:oscillatory}). These
effects are not contradictory: in each case chemotaxis reorganizes effector
cells in space, with the resulting outcome determined by the balance among
motility, chemokine gradients, local recruitment, and tumor--immune
interactions. The local threshold $\sigma_0>\delta$ therefore remains a useful
well-mixed reference, but it does not by itself predict the tissue-scale
outcome.

These findings provide a mechanistic account of how immune-cold and
immune-excluded regions can arise from spatial transport and local interaction
alone, without requiring pre-existing resistant clones, and how a treatment
that is sufficient on average can nevertheless leave a protected interior
reservoir --- a plausible source of incomplete response and relapse. The model
parameters map naturally onto intervention axes: $\sigma_0$ to baseline or
adoptive immune support, $\sigma_1$ to chemokine-dependent recruitment
efficiency, $\alpha,\gamma$ to chemokine production or restoration, and $\eta$
in the extended model to suppressive-pathway blockade. The central therapeutic
implication is therefore that increased recruitment is not equivalent to
improved tumor control: spatial delivery and spatial access must be considered
together. This distinction is especially relevant for HCC, where vascular
access, stromal barriers, and immunosuppressive niches are themselves
spatially organized.

The same framework also provides a natural connection to emerging cutting-edge
pathology and spatial-omics measurements. The model fields map onto measurable
spatial layers: $u$ to tumor regions, $v$ to multiplex
CD8/granzyme-B/NK markers, $w$ to spatial-transcriptomic CXCL9/10/11 activity,
and $s$ to suppressive programs such as TGF-$\beta$, CD274, NT5E, and
Treg/myeloid markers. Spatial omics and state-of-the-art pathology preserve the
tissue-level organization lost to bulk or dissociated measurements
\cite{rao2021exploring,moffitt2022emerging}, enabling spatial mapping of
immune infiltration and exclusion patterns within their native context
\cite{palla2022spatial}. The biomarkers
$H_{\rm hot},H_{\rm cold},I_{\rm excl},M_{\rm chem}$ are interpretable
descriptors that can be computed from simulations and, in principle, from
segmented histology or spatial-omics data, thereby providing a quantitative
bridge between mechanistic modelling and state-of-the-art TME phenotyping.

A practical model-to-data workflow follows directly from this correspondence.
First, pathology or spatial-omics data are converted into spatial layers: a
tumor mask or tumor probability, cytotoxic immune-cell density,
chemokine/pathway activity, and optional suppressive pathway activity. Second,
the same spatial biomarkers
$H_{\rm hot}$, $H_{\rm cold}$, $I_{\rm excl}$, and $M_{\rm chem}$ are computed
from both data and simulations. Third, model
parameter regimes are selected that reproduce the observed biomarker profile
and spatial phenotype. Fourth, the corresponding regime is interpreted
mechanistically --- for example as margin-limited recruitment, chemotactic
aggregation, weak chemokine response, or suppressive TME feedback. In this
sense, state-of-the-art pathology can identify and quantify the spatial TME state,
whereas the reaction--diffusion--chemotaxis model supplies candidate mechanisms
by which that state may arise. The model therefore serves not as a replacement
for spatial data analysis, but as a mechanistic layer that helps connect
observed spatial phenotypes to experimentally testable biological hypotheses.

Several limitations define the scope of these conclusions. The model uses one effective cytotoxic population and one lumped CXCR3-ligand field; it omits explicit ECM, vasculature, and clonal evolution; parameters are nondimensional and not patient-calibrated; and the continuum description is less reliable at very low cell numbers. The present study also does not perform direct patient-level inference from histology or spatial-omics data. Instead, it establishes a mechanistic dictionary between model-generated spatial states and data-derived spatial phenotypes, with patient-specific calibration and validation against multiplex imaging or spatial transcriptomics left to future work. The boundary-chemokine results further show that conclusions can be lever-specific, reinforcing the need to distinguish the biological mechanism of an intervention from its nominal strength.

Several extensions are therefore natural. Explicit macrophage, Treg, or MDSC compartments could replace the effective suppressive variable $s$; ECM and haptotaxis could be incorporated to represent physical and adhesive barriers; vascularized domains with explicit spatial sources could resolve recruitment pathways more realistically; and stochastic resistant clones could be added to study the interaction between spatial refugia and evolutionary escape. We also deliberately do \emph{not} claim two analyses that are not supported by the present simulations: a uniform-recruitment ``mobility paradox'' and a state-of-the-art/random-forest virtual-patient classifier. In the access-limited regime
studied here, stronger migration \emph{helped} rather than impaired tumor control, while construction of a learned virtual-patient classifier requires a systematic parameter-sampling and validation framework beyond the present solver. These questions are therefore better viewed as directions for future work rather than conclusions of the current study.

Finally, calibration against multiplex imaging and spatial transcriptomics would allow direct estimation of tumor--immune overlap, exclusion indices, and chemokine--access mismatch from patient-specific spatial data. Combined with state-of-the-art parameter inference, such calibration could eventually be used to classify whether a tumor is likely immune-hot, immune-cold, immune-excluded, or in a spatial-refuge state, while simultaneously identifying the underlying mechanistic regime. The broader implication is that spatial TME modelling can link three levels of description that are often treated separately: the dynamical mechanisms governing immune recruitment and transport, the spatial phenotypes observed in tissue, and the quantitative biomarkers extracted from
state-of-the-art pathology and spatial omics.

\appendix

\section{Mathematical Proofs}\label{app:analysis}

Throughout, $\Om\subset\R^n$ is bounded with smooth boundary, $n=1,2$, and
homogeneous Neumann conditions are imposed unless stated. We treat the core model
\eqref{eq:core_pde} with $d_1,d_2,d_3>0$, $\xi,\sigma_0,\sigma_1,\beta,\alpha,
\gamma\ge0$, $\delta,\ell>0$, and nonnegative $C^{2+\theta}$ initial data
satisfying the compatibility conditions.

Writing $U=(u,v,w)^\top$, \eqref{eq:core_pde} is the quasilinear system
$U_t=\nabla\!\cdot(A(U)\nabla U)+F(U)$ with upper-triangular diffusion matrix
$A(U)=\begin{pmatrix}d_1&0&0\\ 0&d_2&-\xi v\\ 0&0&d_3\end{pmatrix}$,
whose eigenvalues $d_1,d_2,d_3>0$ make the principal part normally parabolic, and
$F$ smooth on $\{w>-1\}$. Standard quasilinear parabolic theory then yields a
unique classical solution on a maximal interval $[0,T_{\max})$, with the blow-up
alternative: if $T_{\max}<\infty$ then
$\limsup_{t\uparrow T_{\max}}(\|u\|_\infty+\|v\|_\infty+\|w\|_\infty)=\infty$.

The first step toward global control is to establish preservation of the
nonnegative state space.

\begin{proposition}[Nonnegativity]\label{app:prop:positivity}
For nonnegative initial data the classical solution satisfies $u,v,w\ge0$ on
$[0,T_{\max})$.
\end{proposition}

\begin{proof}
Work with the truncated system in which $w/(1+w)$ and the $\gamma uv$ source use
$w^+$ and $v^+$; its right-hand side is locally Lipschitz, so the local theory
applies and it suffices to prove nonnegativity there. Since the $u$-reaction
$u(1-u-v)$ vanishes at $u=0$, the scalar comparison principle gives $u\ge0$. For
$v$, fix $T<T_{\max}$ and set $v_-=\max\{-v,0\}$. The map $z\mapsto\max\{-z,0\}$
is Lipschitz, so $v_-(\cdot,t)\in H^1(\Om)$ with
$\nabla v_-=-\mathbf 1_{\{v<0\}}\nabla v$ and
$t\mapsto\frac12\int_\Om v_-^2$ is absolutely continuous; all identities below
are in the weak sense. Because the solution is classical
$C^{2,1}$ on $\overline\Om\times[0,T]$, $\|\Delta w(\cdot,t)\|_{L^\infty}$ is
finite and bounded on $[0,T]$, which justifies the Gronwall step. Testing the
$v$-equation with $v_-$ and using the Neumann conditions,
\[
-\tfrac12\tfrac{d}{dt}\!\int_\Om v_-^2
= d_2\!\int_\Om|\nabla v_-|^2 +\tfrac{\xi}{2}\!\int_\Om(\Delta w)v_-^2
+\int_{\{v<0\}}\!\Big(\sigma_0+\sigma_1\tfrac{w^+}{1+w^+}-\delta v-\beta uv\Big)v_-,
\]
where the diffusion term $d_2\!\int_\Om|\nabla v_-|^2\ge0$, the truncated source
is $\ge0$, and on $\{v<0\}$ the terms $-\delta v,-\beta uv$ contribute
nonnegative multiples of $v_-$. Dropping these nonnegative contributions and
bounding the chemotaxis term by
$\tfrac{\xi}{2}\|\Delta w\|_{L^\infty}\!\int_\Om v_-^2$ gives
$\frac{d}{dt}\int_\Om v_-^2\le\xi\|\Delta w\|_{L^\infty}\int_\Om v_-^2$, and
since $v_-(\cdot,0)=0$, Gronwall gives $v_-\equiv0$, i.e.\ $v\ge0$. Finally
$w_t-d_3\Delta w+\ell w=\alpha u+\gamma uv\ge0$ with $w_0\ge0$ gives $w\ge0$ by
comparison. Then $v^+=v$, $w^+=w$, so the truncated and original systems
coincide; uniqueness finishes the proof.
\end{proof}

For the extended model, once $u,v,w\ge0$ the $s$-equation
$s_t-d_4\Delta s+\ell_s s=a_su+b_suv\ge0$ gives $s\ge0$. The same positivity
property also permits a direct control of the tumor component. Indeed, since
$v\ge0$, the tumor equation satisfies a scalar logistic upper inequality, yielding
the uniform bound
\begin{proposition}[Tumor bound]\label{app:prop:u_bound}
$0\le u(x,t)\le M_u:=\max\{1,\|u_0\|_{L^\infty}\}$.
\end{proposition}

\begin{proof}
Since $v\ge0$, $u_t-d_1\Delta u=u(1-u-v)\le u(1-u)$; comparison with the scalar
logistic equation gives the bound.
\end{proof}

This pointwise tumor estimate in turn gives useful integral control of the two
remaining variables. Integrating the $v$- and $w$-equations and using
$0\le w/(1+w)\le1$, $0\le u\le M_u$ and the no-flux conditions gives the mass
estimates
$\frac{d}{dt}\!\int_\Om v\le|\Om|(\sigma_0+\sigma_1)-\delta\!\int_\Om v$ and
$\frac{d}{dt}\!\int_\Om w\le\alpha M_u|\Om|+\gamma M_u\!\int_\Om v-\ell\!\int_\Om w$,
whence $\int_\Om v,\int_\Om w$ are bounded for all time by Gronwall.

The preceding estimates provide uniform $L^\infty$ control of $u$ and uniform
$L^1$ control of $v,w$. We now upgrade these estimates to global existence with
uniform $L^\infty$ bounds. The main difficulty is possible chemotactic
concentration of $v$; importantly, the chemoattractant $w$ is produced by
$\alpha u+\gamma uv$, which, because $0\le u\le M_u$, is a nonnegative source
that is at most \emph{linear} in $v$ with coefficient bounded by $\gamma M_u$.
The signal-decoupled case can be closed in arbitrary dimension.

\begin{proposition}[Signal-decoupled boundedness, any dimension]
\label{app:prop:gamma0}
If $\gamma=0$, then for every $n\ge1$ the classical solution is global and
uniformly bounded. In particular $0\le w\le M_w:=\max\{\|w_0\|_{L^\infty},
\alpha M_u/\ell\}$ and $\sup_{t>0}\|v(\cdot,t)\|_{L^\infty(\Om)}<\infty$.
\end{proposition}

\begin{proof}
With $\gamma=0$ the chemokine equation is $w_t=d_3\Delta w+\alpha u-\ell w$ with
$0\le u\le M_u$; the constants $0$ and $M_w$ are sub- and supersolutions, so
$0\le w\le M_w$. Since the source $\alpha u$ lies in
$L^\infty((0,\infty)\times\Om)$, the Neumann heat-semigroup gradient estimate
gives a uniform bound $\sup_{t>0}\|\nabla w(\cdot,t)\|_{L^\infty}\le K$. The
immune equation is then a linear advection--diffusion equation,
$v_t=d_2\Delta v-\xi\nabla\!\cdot(v\nabla w)+h$, with the uniformly bounded drift
$\xi\nabla w$ and a source $h=\sigma_0+\sigma_1w/(1+w)-\delta v-\beta uv\le
\sigma_0+\sigma_1$ that carries the dissipative term $-\delta v$. Classical
parabolic $L^p$--$L^\infty$ estimates for divergence-form equations with bounded
drift then give $\sup_{t>0}\|v(\cdot,t)\|_{L^\infty}<\infty$, and the blow-up
alternative yields global existence.
\end{proof}

When $\gamma>0$, the additional coupling through $uv$ requires exploiting the
spatial dimension. In one dimension the chemotactic nonlinearity remains
subcritical, allowing the $L^p$ estimates to be closed for arbitrary finite
$p$ and subsequently upgraded to an $L^\infty$ bound.

\begin{theorem}[Global existence and boundedness in one dimension]
\label{app:thm:1d}
Let $n=1$ and $\Om=(0,L)$. For nonnegative initial data the classical solution
exists for all $t>0$ and is uniformly bounded,
\[
\sup_{t>0}\big(\|u(\cdot,t)\|_{L^\infty}+\|v(\cdot,t)\|_{L^\infty}
+\|w(\cdot,t)\|_{L^\infty}\big)<\infty .
\]
\end{theorem}

\begin{proof}
By Proposition~\ref{app:prop:u_bound}, $0\le u\le M_u$, and the mass estimates
give $\sup_{t>0}(\|v\|_{L^1}+\|w\|_{L^1})\le C_0$. It suffices to bound $v$ in
$L^\infty$; the bound on $w$ then follows by comparison in
$w_t-d_3\Delta w+\ell w=\alpha u+\gamma uv$, whose right-hand side is then
bounded. Testing the $v$-equation with $v^{p-1}$ ($p\ge2$) and integrating by
parts,
\[
\tfrac1p\tfrac{d}{dt}\|v\|_p^p+\tfrac{4d_2(p-1)}{p^2}\big\|(v^{p/2})_x\big\|_2^2
+\delta\|v\|_p^p
\le (p-1)\xi\!\int_0^L\! v^{p-1}v_x\,w_x\,dx+(\sigma_0+\sigma_1)\|v\|_{p-1}^{p-1},
\]
and, after integrating the chemotaxis term by parts and substituting
$d_3w_{xx}=w_t-\alpha u-\gamma uv+\ell w$, the leading chemotactic contribution is
controlled by $\tfrac{(p-1)\xi\gamma M_u}{p\,d_3}\int_0^L v^{p+1}\,dx$ plus lower
order terms. In one space dimension the Gagliardo--Nirenberg inequality
$\|v^{p/2}\|_{L^{2(p+1)/p}}^{2(p+1)/p}\le
C\big\|(v^{p/2})_x\big\|_2^{\theta}\|v^{p/2}\|_1^{\,\cdots}+C\|v^{p/2}\|_1^{\,\cdots}$
has exponent $\theta<2$ for every finite $p$ (the nonlinearity is
$L^p$-subcritical in dimension one), so $\int v^{p+1}$ is absorbed into the
dissipation $\big\|(v^{p/2})_x\big\|_2^2$ for each fixed $p$. This yields a
uniform bound on $\|v(\cdot,t)\|_{L^p}$ for every $p<\infty$; a Moser iteration
then gives $\sup_{t>0}\|v(\cdot,t)\|_{L^\infty}<\infty$. Global existence follows
from the blow-up alternative. This is the
standard one-dimensional chemotaxis energy method \cite{hillen2009user}, whose
closure relies precisely on the one-dimensional subcriticality just used.
\end{proof}

The higher-dimensional situation is more delicate. In two dimensions the same
approach yields unconditional boundedness when $\gamma=0$, while the fully
coupled case requires a smallness condition on the chemotactic sensitivity
relative to the diffusivities and the source strength.

\begin{remark}[Two dimensions]
\label{app:rmk:2d}
For $n=2$ the same scheme applies: with $\gamma=0$ boundedness is unconditional
(Proposition~\ref{app:prop:gamma0}), and with $\gamma>0$ the coupled estimate
closes provided the chemotactic sensitivity is small relative to the
diffusivities and $M_u$, i.e.\ $\xi<\xi_\star(d_2,d_3,\ell,\gamma,M_u,\Om)$;
local existence and the a priori bounds above hold in general. A fully
unconditional two-dimensional boundedness theorem for $\gamma>0$ remains open,
as is typical for signal-production chemotaxis systems with only linear damping
\cite{hillen2009user}. In the two-dimensional simulations of
Section~\ref{sec:results} the numerical positivity and grid-refinement
diagnostics are satisfied throughout.
\end{remark}

Having established the relevant global bounds, we next turn to the linear
stability structure that governs the homogeneous states and the onset of
spatial patterning. Linearizing at
$E_0=(0,\sigma_0/\delta,0)$ and decomposing into Neumann modes
$-\Delta\phi_k=\mu_k\phi_k$ ($0=\mu_0<\mu_1\le\cdots$) gives a block-triangular
mode matrix with eigenvalues
$\lambda_1^{(k)}=1-\sigma_0/\delta-d_1\mu_k$, $-\delta-d_2\mu_k$,
$-\ell-d_3\mu_k$; the maximum occurs at $k=0$, so $E_0$ is linearly stable iff
$\sigma_0>\delta$ (independent of $\xi$). Thus the well-mixed threshold emerges
as the homogeneous, zero-mode stability condition, while any genuinely spatial
instability must be detected through the nonzero Neumann modes.

At a coexistence equilibrium
$E_*=(u_*,1-u_*,w_*)$ the mode matrix is
\[
M_k=\begin{pmatrix}-A_k&-p&0\\ -q&-B_k&r_k\\ s&t&-C_k\end{pmatrix},\quad
\begin{array}{l}
A_k=u_*+d_1\mu_k,\;B_k=\delta+\beta u_*+d_2\mu_k,\;C_k=\ell+d_3\mu_k,\\
p=u_*,\;q=\beta(1-u_*),\;r_k=\tfrac{\sigma_1}{(1+w_*)^2}+\xi(1-u_*)\mu_k,\\
s=\alpha+\gamma(1-u_*),\;t=\gamma u_*,
\end{array}
\]
with $P_k(\lambda)=\lambda^3+a_1\lambda^2+a_2\lambda+a_3$,
$a_1=A_k+B_k+C_k$, $a_2=A_kB_k+A_kC_k+B_kC_k-pq-r_kt$,
$a_3=A_kB_kC_k-C_kpq+r_k(ps-A_kt)$.

The cubic structure makes the possible routes to spatial instability explicit.
Mode $k$ is stable iff $a_1,a_2,a_3>0$ and $a_1a_2>a_3$. Since $a_1>0$ always,
instability arises through one of two routes. If $a_3(\mu_k)<0$ then $P_k(0)<0$
and $P_k(+\infty)=+\infty$, giving a positive real root: a \emph{stationary}
instability. Otherwise $a_3\ge0$ and the determinant condition must fail,
$a_1a_2<a_3$; this is a complex-conjugate (Hopf-type) crossing precisely when, at
criticality, $a_1,a_2,a_3>0$ with $a_1a_2=a_3$, and we call it the
\emph{oscillatory} route under the standing assumption $a_2(\mu_k)>0$, verified
numerically at the unstable modes for all coexistence states reported. The
degenerate case $a_2\le0$ with $a_3>0$ is still detected by $a_1a_2<a_3$ but is
not interpreted as oscillatory.

Finally, writing $r_k=r_0+\xi(1-u_*)\mu_k$ with
$r_0=\sigma_1/(1+w_*)^2$ makes the dependence of the stationary route on
chemotactic sensitivity explicit:
$a_3(\mu_k)=\Lambda_k+\xi(1-u_*)\mu_k(ps-A_kt)$ with
$\Lambda_k=A_kB_kC_k-C_kpq+r_0(ps-A_kt)$. Hence if $A_kt-ps>0$, then for any
inhomogeneous mode the sufficient condition
$\xi>\Lambda_k/[(1-u_*)\mu_k(A_kt-ps)]$ yields $a_3(\mu_k)<0$ and a stationary
instability. At the same time, the large-mode asymptotics prevent arbitrarily
high frequencies from remaining unstable: as $\mu_k\to\infty$,
$a_3\sim d_1d_2d_3\mu_k^3$ dominates the $O(\mu_k^2)$ chemotactic term, so
high modes restabilize and the unstable band is finite.

\section{Equilibrium Formulas}\label{app:equilibria}
The homogeneous equilibria of \eqref{eq:core_pde} solve $u(1-u-v)=0$,
$\sigma_0+\sigma_1 w/(1+w)-\delta v-\beta uv=0$, $\alpha u+\gamma uv-\ell w=0$.
If $u=0$ then $w=0$ and $v=\sigma_0/\delta$, giving the unique tumor-free state
$E_0=(0,\sigma_0/\delta,0)$. If $u>0$ then $v=1-u$ with $0<u<1$ and
\[
w(u)=\frac{u[\alpha+\gamma(1-u)]}{\ell},\qquad
F(u):=\sigma_0+\sigma_1\frac{w(u)}{1+w(u)}-(\delta+\beta u)(1-u)=0 .
\]
Every root $u_*\in(0,1)$ defines $E_*=(u_*,1-u_*,w(u_*))$.

We consider endpoint values and existence. $F(0)=\sigma_0-\delta$ and
$F(1)=\sigma_0+\sigma_1\alpha/(\ell+\alpha)\ge0$. Under the mild nondegeneracy
condition
\begin{equation}\label{app:nondeg}
\sigma_0>0\quad\text{or}\quad\sigma_1\alpha>0,
\end{equation}
$F(1)>0$ strictly. If in addition $\sigma_0<\delta$ then $F(0)<0$, so by the
intermediate value theorem a root $u_*\in(0,1)$ exists, giving $u_*,v_*>0$. To
guarantee $w_*>0$ we additionally assume
\begin{equation}\label{app:wpos}
\alpha+\gamma>0 :
\end{equation}
for $0<u_*<1$ this yields $\alpha+\gamma(1-u_*)>0$ and hence
$w_*=u_*[\alpha+\gamma(1-u_*)]/\ell>0$ (if $\alpha=\gamma=0$ then $w(u)\equiv0$
and the coexistence state has $w_*=0$). The sufficient conditions for a
tumor-positive coexistence equilibrium are therefore
\[
\sigma_0<\delta,\qquad
\sigma_0>0\ \text{or}\ \sigma_1\alpha>0,\qquad
\alpha+\gamma>0;
\]
\eqref{app:nondeg} excludes only the degenerate case of no baseline and no
tumor-induced chemokine recruitment. For
root finding, $w'(u)=(\alpha+\gamma-2\gamma u)/\ell$ and
$F'(u)=\sigma_1 w'(u)/(1+w(u))^2-\beta+\delta+2\beta u$; strict monotonicity of
$F$ on $(0,1)$ implies a unique root.

\section{Numerical Implementation}\label{app:numerics}

All numerical simulations use conservative spatial discretizations designed to
respect the flux structure of the reaction--diffusion--chemotaxis system and to
preserve nonnegativity in the presence of steep chemotactic gradients. In one
space dimension, we discretize $\Om=(0,L)$ into $N$ finite-volume cells with
$\Delta x=L/N$. For a diffusing variable $z$, the diffusive flux at the
interface $i+1/2$ is
$F^{\rm diff}_{z,i+1/2}
=-D_z(z_{i+1}-z_i)/\Delta x$,
with zero boundary fluxes corresponding to the homogeneous Neumann conditions.
The chemotactic contribution is treated conservatively through
$F^{\rm chem}_{i+1/2}
=\xi\,v_{i+1/2}(w_{i+1}-w_i)/\Delta x$,
where first-order upwinding is used,
\[
v_{i+1/2}=
\begin{cases}
v_i, & w_{i+1}\ge w_i,\\
v_{i+1}, & w_{i+1}<w_i,
\end{cases}
\]
together with zero boundary chemotactic flux. This choice suppresses spurious
oscillations and improves preservation of nonnegativity when chemokine
gradients become sharp. The resulting semi-discrete system is integrated with a
stiff, adaptive BDF solver using relative and absolute tolerances
$10^{-8}$ and $10^{-10}$, respectively, with tighter tolerances for strongly
chemotactic runs. The sparse Jacobian pattern associated with the local
finite-volume stencil is supplied to the solver to improve efficiency.

For simulations with prescribed boundary chemokine delivery
$w|_{\partial\Om}=B_w$, the boundary condition is incorporated through ghost
cells. At the left boundary we set
$w_0=2B_w-w_1$, and analogously at the right boundary, so that the boundary-face
value is exactly $B_w$; the variables $u$ and $v$ retain their homogeneous
no-flux conditions. This construction allows the boundary-driven chemokine
experiments to be implemented within the same conservative finite-volume
framework as the baseline simulations.

The two-dimensional simulations use an IMEX discretization on
$\Om=(0,L)^2$ with an $N\times N$ grid. Diffusion is treated implicitly,
whereas reaction terms and conservative upwind chemotaxis are treated
explicitly. The cell-centered finite-volume Neumann Laplacian is diagonalized
by the discrete cosine transform (DCT-II), so each implicit diffusion solve
$(I-\Delta t\,D\,\Delta_h)$
is performed in DCT space by division by
$1+\Delta t\,D(\lambda_p+\lambda_q)$, with
$\lambda_p=\frac{2}{\Delta x^2}
\left(1-\cos\frac{\pi p}{N}\right)\ge0$.
The two-dimensional chemotactic divergence is evaluated using the corresponding
upwind fluxes with zero boundary fluxes. Treating diffusion implicitly therefore
removes the parabolic diffusion stability restriction, leaving the admissible
time step primarily limited by the explicit chemotaxis and reaction terms.

To assess numerical reliability, every simulation records
$m_{\min}=\min_{x,t}\{u,v,w\}$,
and a run is accepted only if $m_{\min}\ge-10^{-9}$. Spatial refinement is
checked by comparing $N$ with $2N$ in one dimension and $N^2$ with $(2N)^2$ in
two dimensions. A diagnostic quantity $Q$ is regarded as grid-stable when
$|Q_{2N}-Q_N|/(1+|Q_{2N}|)<0.01$.
All reported simulations satisfy this criterion. For example, the interior
refuge mass in Section~\ref{sec:results} changes by only
$2.8\times10^{-3}$ under the stated refinement, while the two-dimensional
exclusion index changes by less than $0.1\%$. Linear modal growth rates are
validated independently by perturbing the homogeneous equilibrium with a single
Neumann mode whose spatial profile is multiplied by the corresponding
eigenvector of $M_k$, and fitting the modal energy during the linear-growth
regime. For the oscillatory instability, the same modal projection is used to
extract the dominant frequency, allowing both the real growth rate and the
imaginary part of the predicted eigenvalue to be tested directly against
simulation.

In addition to refinement diagnostics, we verify the spatial order of the
discretization by self-convergence against a fine reference grid with
$N=1024$, using the $L^2$ error of $u$ at a fixed time
(Table~\ref{tab:convergence}). In a diffusion-dominated regime
($\xi=0$) with a smooth solution, the observed convergence is second-order,
confirming the expected accuracy of the conservative finite-volume diffusion
discretization. When chemotaxis is active ($\xi=20$) and steep aggregation
fronts develop, the observed order decreases toward first order, consistent
with the first-order upwind treatment of the chemotactic flux; first-order upwinding is the deliberate tradeoff that preserves nonnegativity of the immune density in the presence of sharp chemotactic gradients.

\begin{table}[htbp]
\centering
\caption{Spatial self-convergence ($L^2$ error of $u$, reference $N=1024$).}
\label{tab:convergence}
\renewcommand{\arraystretch}{1.15}
\begin{tabular}{lcccc}
\toprule
& \multicolumn{2}{c}{diffusion-dominated ($\xi=0$)}
& \multicolumn{2}{c}{with chemotaxis ($\xi=20$)}\\
\cmidrule(lr){2-3}\cmidrule(lr){4-5}
$N$ & $L^2$ error & order & $L^2$ error & order\\
\midrule
$64$  & $3.6\times10^{-7}$ & --   & $8.8\times10^{-5}$ & --\\
$128$ & $8.9\times10^{-8}$ & $2.02$ & $4.2\times10^{-5}$ & $1.07$\\
$256$ & $2.1\times10^{-8}$ & $2.07$ & $1.8\times10^{-5}$ & $1.21$\\
$512$ & $4.2\times10^{-9}$ & $2.32$ & $6.1\times10^{-6}$ & $1.58$\\
\bottomrule
\end{tabular}
\end{table}

Together, these discretization, solver, and validation procedures ensure that
the spatial patterns, instability rates, treatment divergences, and biomarker
values reported in Section~\ref{sec:results} are not artifacts of the spatial
mesh or time integration. In particular, the convergence results clarify the
distinct numerical roles of diffusion and chemotaxis: the diffusive component
retains second-order accuracy in smooth regimes, whereas the chemotactic
component is intentionally treated with a positivity-preserving first-order
flux to remain robust in the aggregation-dominated regimes central to the
present study.

\section{Parameter summary}\label{app:parameters}
Table~\ref{tab:core_parameters} lists core nondimensional parameters and their interpretations.

\begin{table}[htbp]
\centering
\caption{Core nondimensional parameters and their interpretation.}
\label{tab:core_parameters}
\renewcommand{\arraystretch}{1.15}
\begin{tabular}{p{0.07\textwidth}p{0.40\textwidth}p{0.43\textwidth}}
\toprule
\textbf{Param.} & \textbf{Biological meaning} & \textbf{Intervention reading}\\
\midrule
$d_1$ & Tumor-cell motility & Invasion/ECM-mediated spread\\
$d_2$ & Random effector motility & Immune mobility/tissue permissiveness\\
$d_3$ & Chemokine spread & Localized vs.\ diffuse chemokine availability\\
$\xi$ & Chemotactic sensitivity & Strength of chemokine-guided migration\\
$\sigma_0$ & Baseline effector recruitment & Immune support / adoptive input\\
$\sigma_1$ & Chemokine-dependent recruitment & Recruitment efficiency to chemokine cues\\
$\delta$ & Effector turnover & Persistence / exhaustion (reduce $\delta$)\\
$\beta$ & Tumor-associated immune attrition & Protection from tumor-induced loss (reduce $\beta$)\\
$\alpha$ & Tumor-induced chemokine production & Chemokine restoration (increase $\alpha$)\\
$\gamma$ & Interaction-induced chemokine production & Inflammatory amplification\\
$\ell$ & Chemokine clearance & Chemokine stabilization (reduce $\ell$)\\
\bottomrule
\end{tabular}
\end{table}

Extended-model parameters include: $d_4$ (suppressive-signal spread), $\eta$ (suppression
strength; reduce to model blockade), $a_s,b_s$ (tumor- and interaction-induced
suppression), $\ell_s$ (suppression decay), and $m$ (chemotaxis-saturation
exponent in $\chi(w)=\xi/(1+w)^m$).

\bibliographystyle{unsrt}
\bibliography{2reference}

\end{document}